\documentclass[12pt, reqno]{amsart}%
\usepackage{amsmath, amsthm, amscd, amsfonts, amssymb, graphicx, color}
\usepackage[bookmarksnumbered, colorlinks, plainpages]{hyperref}
\usepackage{amsmath}
\usepackage{amsfonts}
\usepackage{amssymb}
\usepackage{graphicx}%
\usepackage{multicol}
\providecommand{\U}[1]{\protect\rule{.1in}{.1in}}
\makeatletter
\@namedef{subjclassname@2020}{	\textup{2020} Mathematics Subject Classification}
\makeatother
\newtheorem{theorem}{Theorem}[section]
\newtheorem{lemma}[theorem]{Lemma}
\newtheorem{proposition}[theorem]{Proposition}

\theoremstyle{definition}
\newtheorem{definition}[theorem]{Definition}

\newtheorem{remark}[theorem]{Remark}
\numberwithin{equation}{section}

\newcommand{\be}{\begin{equation}}
\newcommand{\ee}{\end{equation}}
\begin{document}
\title[A Family of Iterated Maps on Natural Numbers]{A Family of Iterated Maps on Natural Numbers}
\author [A. DAS]{ANGSUMAN DAS}

\address{Department of Mathematics, Presidency University, Kolkata, India}
\email{\textcolor[rgb]{0.00,0.00,0.84}{angsuman.maths@presiuniv.ac.in}}

\begin{abstract}
In this paper, we introduce and study the iterates of the following family of functions $\varphi_k$ defined on natural numbers which exhibits nice properties. 
$$\varphi_k(x)=\left\lbrace \begin{array}{ll}
x+k, & \mbox{ if $x$ is prime;}\\
\mbox{largest prime divisor of $x$,} & \mbox{ if $x$ is composite;}
\end{array} \right.$$
In particular, we study the periodic behaviour of the trajectories of these iterated functions. In some cases, we provide proofs of these properties and in some other cases we pose some open problems based on numerical evidences supported by heuristic arguments.
\end{abstract}
\maketitle
\setcounter{page}{1}
\section{Introduction}
The dynamics of arithmetic functions \cite{silverman-book} and in particular, iteration of functions on natural numbers recieved lot of attention since the inception of Collatz function \cite{coll-func-1} and the conjecture regarding Collatz iterates. Since then various generalizations of Collatz function \cite{3x+1} and other constructions \cite{similar-1} of iterated maps appeared in literature. Most of these maps are either proved or conjectured to be eventually periodic on when iterated on any given natural number. A well-known survey and a recent remarkable result on Collatz conjecture can be found here \cite{collatz-book} and \cite{tao-paper}.

In this paper, we introduce a family of functions on natural numbers which exhibits nice periodic properties. While some of the properties of this family of maps are proved, some others has been supported by heuristic arguments. Before describing the family, we recall a few definitions in connection with iterated maps.

\begin{definition}
	Let $f:X \rightarrow X$ be a function.
	\begin{itemize}
		\item For a given $x \in X$, the sequence of values $\{f^n(x)\}=\{x,f(x),f^2(x),\ldots,$ $f^n(x),\ldots \}$ is called the \textit{orbit} or \textit{trajectory} of $x$.
		\item An orbit $\{f^n(x)\}$ is said to be \textit{eventually periodic} or \textit{periodic} if there exists $m,l \in \mathbb{N}$ such that $f^i(x)=f^{i+l}(x)$ for all $i\geq m$.
		\item The smallest such value of $l$ for a given $x$ is called the \textit{period} of the orbit and the point $x$ itself is called a \textit{periodic point}. Any consecutive $l$ terms in $\{f^m(x),f^{m+1}(x),f^{m+2}(x),\ldots, \}$ is called a \textit{loop} or \textit{cycle} in that eventually periodic sequence.
		\item The \textit{stopping time} of $x$ w.r.t $f$ is the largest $n$ such that $x,f(x),f^2(x),\ldots,$ $f^n(x)$ are all distinct.
	\end{itemize}
	
\end{definition}
Now we define the family of functions which is the main topic of interest of this paper.

\begin{definition}
Let $k \in \mathbb{N}$. Define a function $\varphi_k:\mathbb{N}\setminus \{1\} \rightarrow \mathbb{N}\setminus \{1\}$ by $$\varphi_k(x)=\left\lbrace \begin{array}{ll}
x+k, & \mbox{ if $x$ is prime;}\\
\mbox{largest prime divisor of $x$,} & \mbox{ if $x$ is composite;}
\end{array} \right.$$

Also define $S(x_0,k)$ to be the integer sequence $(x_n)$ with first term $x_0$ and subsequent terms are defined by the recursion $x_i=\varphi_k(x_{i-1})$, i.e., $$S(x_0,k)=\{x_0,\varphi_k(x_0),\varphi^2_k(x_0),\varphi^3_k(x_0),\varphi^4_k(x_0),\ldots,,\varphi^n_k(x_0),\ldots \}$$
\end{definition}

We now see some examples which indicates some periodic behaviour of such sequences.
	\begin{itemize}
		\item $S(8,2)=\{8,\mathbf{2,4},2,4,2,4,2,\ldots \}$ is of period $2$ and stopping time $3$.
		\item $S(5,1)=\{5,6,\mathbf{3,4,2},3,4,2,3,4,\ldots \}$ is of period $3$ and stopping time $5$.
		\item $S(2,12)=\{2,14,\mathbf{7,19,31,43,55,11,23,35},7,\ldots \}$ is of period $8$ and stopping time $10$.
	
		\item $S(7,15)=\{\mathbf{7, 22, 11, 26, 13, 28}, 7, 22, 11 \ldots \}$ is of period $6$.
		\item $S(17,15)=\{\mathbf{17, 32, 2}, 17, 32, \ldots \}$ is of period $3$.
	\end{itemize}  The bold entries denote a loop in the corresponding sequence.

We note that $S(x_0,k)$ can have at most $k$ consecutive terms to be prime, because, for any prime $p$, at most $p,p+k,p+2k,\ldots,p+(p-1)k$ can be primes, but $p+pk=(p+1)k$ is composite. So, we must have a composite term within the first $(k+1)$ terms of $S(x_0,k)$. And once we have a composite term, the next term must be a prime. As periodicity of a sequence does not change by deleting first finitely many terms, without loss of generality, it is enough to study the sequences $S(p,k)$, where the first term of the sequence varies over the set of primes. 

\section{Eventual Periodicity, Stopping Time and Loops}
In this section , we discuss the eventual periodicity, stopping time  and loops of $S(p,k)$. In particular, we show that 
\begin{itemize}
	\item $S(p,k)$ is eventually periodic for all $p$ and all $k$.
	\item For suitable choices of $p$ and $k$, the stopping time of $S(p,k)$ can be  arbitrarily large.
	\item For a fixed $k$, the number of ditinct loops in $S(x,k)$ is finite, as $x$ varies.
\end{itemize}

Before the main proofs, we recall a definition and prove few useful lemmas for our aid. For a positive integer $n>1$, the primorial of $n$, denoted by $n\#$, is defined as the product of all primes less than or equal to $n$.

\begin{lemma}\label{primorial-lemma}
	If $p, p+k, p+2k, \ldots,p+(l-1)k$ is an arithmetic progression of $l$ primes with common difference $k$, then $k$ is a multiple of $(l-1)\#$.
\end{lemma}
\begin{proof}
	Let $q$ be a prime less than $l$.  If $k$ is not divisible by $q$, then $k$ is a generator of the group $\mathbb{Z}_q$, i.e., in modulo $q$, $\{k,2k,\cdots, qk\}=\mathbb{Z}_q$. Thus there exists $t \in \{1,2,\ldots,q\}$ such that $p+tk$ is divisible by $q$. As $1\leq t\leq q< l$, $p+tk$ is also a prime, a contradiction. Thus $k$ is divisible $q$, and hence is divisible by the product of all primes less than $l$. 
\end{proof}

\begin{lemma}\label{odd-lemma}
Let $k\geq 3$ be an odd integer and $p$ be a prime with $p>k$. Then $\varphi^2_k(p)<p$. 
\end{lemma}
\begin{proof}
	As $p$ is prime and $k$ is odd, we have $\varphi_k(p)=p+k$ is even and hence composite. Thus $\varphi^2_k(p)$ is the largest prime divisor of $p+k$. We show that it is less than $p$. Let $k=2t+1$. As $p$ is odd, all of $p+1,p+3,\ldots,p+(2t-1)$ are even and none of them is prime. Also $p+2t$ does not divide $p+k$. If possible, $\varphi^2_k(p)\geq p$. Then $\varphi^2_k(p)$ must be of the form $p+2\alpha$ where $0\leq \alpha\leq t-1$. However, as $(p+2\alpha)|(p+2t+1)$, we must have $2(p+2\alpha)\leq p+2t+1$, i.e., $p\leq 2t-4\alpha+1\leq 2t+1=k$, a contradiction. Thus $\varphi^2_k(p)<p$.
\end{proof}

\begin{lemma}\label{even-lemma}
	Let $k$ be even and $p$ be a prime with $p>k^2/2$. Then there exists $s \in \mathbb{N}$ such that $\varphi^s_k(p)$ is a prime less than $p$.
\end{lemma}
\begin{proof}
	If $\varphi_k(p)=p+k$ is composite, then as $p+k$ is odd, its largest prime factor is $\leq (p+k)/3$, i.e., $\varphi^2_k(p)\leq \frac{p+k}{3}$. Now $p>k^2/2$ implies $p>k/2$, i.e., $p>\frac{p+k}{3}$, i.e., $\varphi^2_k(p)<p$ and the lemma holds.
	
	So, we assume that $p,p+k,p+2k,\ldots,p+(l-1)k$ are all primes and $p+lk$ be composite, i.e., $\varphi^{l-1}_k(p)=p+(l-1)k$, $\varphi^l_k(p)=p+lk$ and $\varphi^{l+1}_k(p)=(p+lk)/3$. Note that we must have $l\leq p$.
	
	As $p, p+k, p+2k, \ldots,p+(l-1)k$ is an arithmetic progression with $l$ terms and common difference $k$, by Lemma \ref{primorial-lemma}, $k$ is a multiple of $(l-1)\#$, i.e., $k\geq (l-1)\#\geq l$.
	
	Now, $p>k^2/2$ implies $2p>k^2\geq kl$, i.e., $\frac{p+lk}{3}<p$, i.e., $\varphi^{l+1}_k(p)=(p+lk)/3<p$.
\end{proof}

Combining the above Lemma \ref{odd-lemma} and \ref{even-lemma}, we get the following proposition.
\begin{proposition}\label{combined-proposition}
	Let $p$ be an odd prime and $k$ be a positive integer such that $p>k^2/2$. Then there exist $s \in \mathbb{N}$ such that $\varphi^s_k(p)$ is a prime less than $p$.
\end{proposition}

Now, we are in a position to prove the main theorems of this section.
\begin{theorem}
	Let $p$ be a prime and $k$ be a positive integer. Then $S(p,k)$ is eventually periodic.
\end{theorem}
\begin{proof}
	First, we assume that $p$ is an odd prime and prove two claims which will be used in the proof of the theorem.
	
	{\it Claim 1:} $S(p,k)$ has a prime entry in the range $[2,k^2/2]$.\\
	{\it Proof of Claim 1:} If $p<k^2/2$, then the first term of $S(p,k)$ is $p$ and hence it lies in the range $[2,k^2/2]$. If $p>k^2/2$, then by above proposition, there exists $s_1 \in \mathbb{N}$ such that $\varphi^{s_1}_k(p)$ is a prime, say $p_1<p$. If $p_1<k^2/2$, then the claim holds. If $p_1>k^2/2$, then again applying the above proposition, we get $s_2 \in \mathbb{N}$ such that $\varphi^{s_1+s_2}_k(p)=\varphi^{s_2}_k(p_1)$ is a prime, say $p_2<p_1$. Continuing in this way, either we get a prime entry of $S(p,k)$ in $[2,k^2/2]$ or we get infinitely many primes with $p>p_1>p_2>\cdots >p_n>\cdots >k^2/2$. Since, there are only finitely many primes in $[k^2/2,p]$, the later can not hold, thereby proving the claim.
	
	As deleting first finitely many terms does not affect the eventual periodicity of a sequence, from Claim 1, we can assume that $p<k^2/2$.
	
	{\it Claim 2:} If all the prime entries of $S(p,k)$ are in $[2,k^2/2]$, then $S(p,k)$ is eventually periodic.\\
	{\it Proof of Claim 2:} If all the prime entries of $S(p,k)$ are in $[2,k^2/2]$, then exactly one of the two cases mentioned below can occur:
	\begin{enumerate}
		\item $S(p,k)\subseteq [2,k^2/2]$: An infinite iterated sequence taking value from a finite set must be eventually periodic.
		\item The prime entries of $S(p,k)$ lies in $[2,k^2/2]$ and $S(p,k)$ admits only composite values outside $[2,k^2/2]$ with all of those composite entries having their largest prime factors in $[2,k^2/2]$: In this case, as $[2,k^2/2]$ have finitely many primes, $S(p,k)$ must be eventually periodic.
	\end{enumerate}
Let, if possible, $S(p,k)$ is not eventually periodic. Then by Claim 2, $S(p,k)$ takes some prime entries outside $[2,k^2/2]$.  Thus, by the above proposition, some term of $S(p,k)$ has to re-enter $[2,k^2/2]$ with a prime entry, say $p_1<p$. Again, arguing as above, if all the prime entries of $S(p_1,k)$ are in $[2,k^2/2]$, then $S(p_1,k)$ and hence $S(p,k)$ is eventually periodic, a contradiction. Thus, $S(p_1,k)$ takes some prime entries outside $[2,k^2/2]$ and it has to re-enter $[2,k^2/2]$ with a prime entry, say $p_2<p_1$. Continuing in this way, we get distinct primes $p_1,p_2,p_3,\ldots \in [2,k^2/2]$. However, as $[2,k]$ has finitely many primes, this leads to a contradiction. 

Thus the theorem holds for all odd primes $p$. If $p=2$ and $k=1$ or $2^n-2$, we get the loops $2,3,4$ and $2,2^n$ respectively. For $p=2$ and other $k$'s, $\varphi_k(2)=2+k$ is either a prime $>2$ or a composite with largest prime factor $>2$. In any case, it is eventually periodic by the above case of odd primes.  
\end{proof}

\begin{theorem}
	 For all $l\geq 2$, there exists a prime $p$ and $k \in \mathbb{N}$ such that $S(x,k)$ has stopping time $>l$.
\end{theorem}
\begin{proof}
	The Green-Tao theorem \cite{green-tao} states that the sequence of prime numbers contains arbitrarily long arithmetic progressions. In other words, for every natural number $l$, there exist arithmetic progressions of primes with $l$ terms. Let the arithmetic progression starts with $p$ and common difference be $k$. Thus $S(p,k)$ has the first $l$ times as distinct primes. As $S(p,k)$ is eventually periodic, its stopping time is greater than $l$.
\end{proof}

Let us recall that $S(7,15)$ eventually enters the loop $7,22,11,26,13,28$ and $S(17,15)$ eventually enters the loop $17,32,2$. Thus for a fixed $k$, $S(x,k)$ may enter different loops if we vary $x$. So, the natural question is: For a fixed $k$, how many distinct loops are there in $S(x,k)$ where $x$ varies over $\mathbb{N}\setminus \{1\}$? More specifically, is the number of such possible distinct loops finite? In the next theorem, we answer this question affirmatively.

\begin{theorem}
Let $k \in \mathbb{N}$ be fixed. The number of distinct loops in $S(x,k)$ as $x$ varies is finite.
\end{theorem}
\begin{proof}
	As mentioned earlier it is enough to vary $x$ over the set of primes. First, we consider the primes $p>k^2/2$. By Proposition \ref{combined-proposition}, there exists $s_1 \in \mathbb{N}$ such that $\varphi^{s_1}_k(p)=p_1<p$, where $p_1$ is a prime. If $p_1>k^2/2$, we apply the same proposition to get $s_2$ such that $\varphi^{s_1+s_2}_k(p)=\varphi^{s_2}_k(p_1)=p_2<p_1<p$. As the number of primes in $[k^2/2,p]$ is finite, this process terminates finitely. Thus, for any prime $p>k^2/2$, there exists $s_p \in \mathbb{N}$ (depending upon $p$) such that $\varphi^{s_p}_k(p)$ is a prime in $[2,k^2/2]$.
	
	Now, there are finitely many, say $t$ primes in $[2,k^2/2]$  and as each loop contains at least one of those $t$ primes, we can have at most $t$ distinct loops. Hence the theorem holds.
\end{proof}

\begin{remark}
	It is to be noted that for a fixed $k$, any loop must have at least one prime entry less than $k^2/2$. Thus the number of loops is bounded above by the number of primes less than $k^2/2$. In fact, we can make this a little tighter. In the proof of Lemma \ref{even-lemma}, we used the approximation, that $(l-1)\#\geq l$ for $l\geq 3$. However, for $l\geq 7$, we have $(l-1)\#\geq l^2$. By separately considering the cases, when $l\leq 6$ and using $(l-1)\#\geq l^2$ for rest of the cases, it can be shown that any loop must take at least one prime entry less than $k\sqrt{k}/2$. Thus the number of loops for a given $k$ is bounded above by the number of primes $\leq k\sqrt{k}/2$.
\end{remark}	
	
\section{Open Issues}
In the last remark, we have seen that for a given $k$, the number of distinct loops in $S(x,k)$ is finite. However, if we vary $k$, the number of distinct loops of $S(x,k)$ seems to be unbounded above. This is demonstrated in Figure \ref{k-vs-loops}. For $k=1\ldots 5000$, we plot the number of distinct loops. It is evident from the graph that smaller the number of loops, it is achieved by more values of $k$. In this range, the maximum number of loops $(14)$ is achieved only by $k=4479$.

\begin{figure}[ht]
	\centering
	\begin{center}
		\includegraphics[scale=0.6]{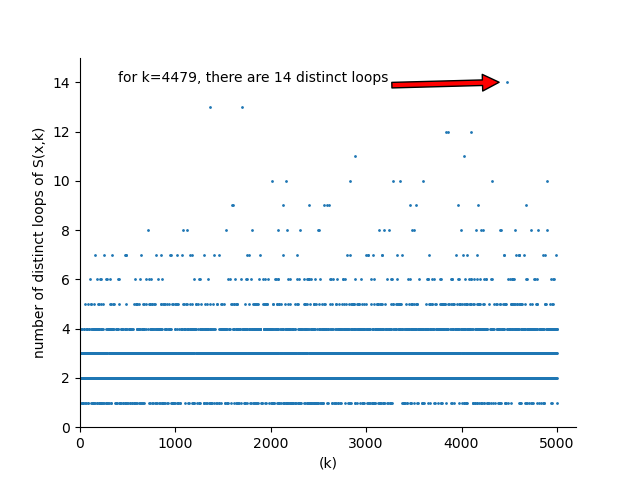}
		\caption{The $k$-vs-no.of.loops plot}
		\label{k-vs-loops}
	\end{center}
\end{figure}

Another property which is observed is that periods of all lengths are attained by some sequence. Figure \ref{l-vs-k} below shows the least value of $k$ for which a sequence of period $l=1\ldots 50$ is obtained. It is to be noted that $l=49$ is first achieved for $k=1428$. This indicates that the least value of $k$ increases sharply for small increase in $l$. 
\begin{figure}[ht]
	\centering
	\begin{center}
		\includegraphics[scale=0.6]{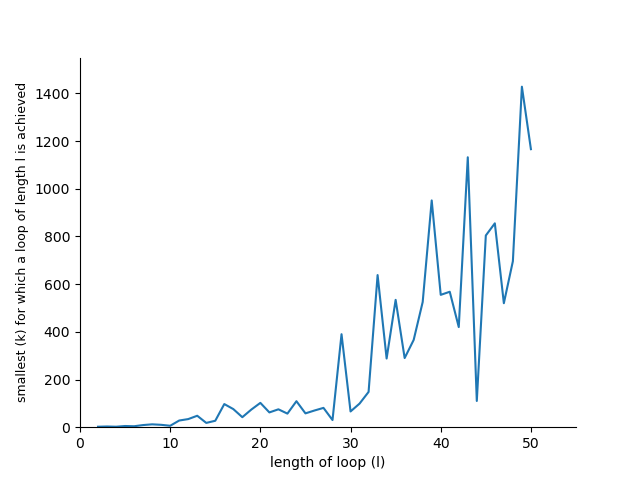}
		\caption{The $l$-vs-$k$ plot}
		\label{l-vs-k}
	\end{center}
\end{figure}
Based on the above observations regarding numerical evidences, we conclude this paper with two open issues:

\begin{enumerate}
	\item For any $m \in \mathbb{N}$, there exists $k\in \mathbb{N}$ such that number of distinct loops of $S(x,k)$ is greater than $m$.
	\item For all $l\geq 2$, there exists a prime $p$ and $k \in \mathbb{N}$ such that $S(p,k)$ is eventually periodic with period $l$.
\end{enumerate}

\section*{Acknowledgements}
The author acknowledge the funding of DST-FIST Sanction no. $SR/FST/MS-I/2019/41$ and DST-SERB-MATRICS Sanction no. $MTR/2022/000020$, Govt. of India. The author is also thankful to Hiranya Kishore Dey and Anubrato Bhattacharya for various fruitful discussions, and the weekly departmental seminar series QED of Presidency University for providing a stimulus to think about this problem.


\begin{thebibliography}{99}                                                              %
\bibitem{coll-func-1} C.J. Everett, Iteration of the number-theoretic function $f(2n) = n, f(2n + 1) = 3n + 2$, \textit{Advances in Mathematics} {\bf 25(1)}, pp.42--45, 1977.
\bibitem{green-tao} B. Green and T. Tao, The primes contain arbitrarily long arithmetic progressions, \textit{Annals of Mathematics}, 167(2), pp. 481--547, 2008.
\bibitem{similar-1} D.A. Klarner and R. Rado, Arithmetic properties of certain recursively defined sets, \textit{Pacific J. Math.} 53, No. 2, pp. 445--463, 1974.
\bibitem{3x+1} J.C. Lagarias, The $3x + 1$ problem and its generalizations, \textit{Amer. Math. Monthly} 92, pp. 3--23, 1985.
\bibitem{collatz-book} J.C. Lagarias (Ed.), The Ultimate Challenge: The $3x+1$ Problem, \textit{American Mathematical Society}, 344pp, 2010.
\bibitem{silverman-book} J.H. Silverman, The Arithmetic of Dynamical Systems, \textit{Springer Graduate Texts in Mathematics}, 241, 2007.
\bibitem{tao-paper} T. Tao, Almost all orbits of the Collatz map attain almost bounded values, \textit{Forum Math. Pi}, 10, Paper No. e12, 56 pp, 2022.
\end{thebibliography}
\end{document}